\documentclass[a4paper,UKenglish,cleveref,autoref,thm-restate]{lipics-v2021}

\hideLIPIcs
\nolinenumbers

\title{Conflict-Free Colouring of Subsets}
\author{Bruno Jartoux}{Ben-Gurion University of the Negev, Be'er-Sheva, Israel}{jartoux@post.bgu.ac.il}{https://orcid.org/0000-0002-5341-1968}{Research partially supported by the Israel Science Foundation (grant no. 1065/20).}
\author{Chaya Keller}{Ariel University, Ariel, Israel}{chayak@ariel.ac.il}{https://orcid.org/0000-0001-6400-3946}{Research partially supported by the Israel Science Foundation (grant no. 1065/20).}
\author{Shakhar Smorodinsky}{Ben-Gurion University of the Negev, Be'er-Sheva, Israel}{shakhar@math.bgu.ac.il}{https://orcid.org/0000-0003-3038-6955}{Research partially supported by the Israel Science Foundation (grant no. 1065/20).}
\author{Yelena Yuditsky}{Université libre de Bruxelles, Brussels, Belgium }{yuditskyL@gmail.com}{https://orcid.org/0000-0002-6467-3437}{Research supported by the Belgian National Fund for Scientific Research (FNRS), through PDR grant BD-OCP.}
\authorrunning{B. Jartoux, C. Keller, S. Smorodinsky, and Y. Yuditsky}
\Copyright{Bruno Jartoux, Chaya Keller, Shakhar Smorodinsky, and Yelena Yuditsky}

\ccsdesc{Mathematics of computing~Hypergraphs}
\ccsdesc{Mathematics of computing~Graph coloring}
\ccsdesc{Theory of computation~Computational geometry}
\keywords{conflict-free, hypergraph colouring, geometric hypergraphs}
\usepackage{textgreek}
\usepackage{todonotes}
\usepackage{stmaryrd} 
\usepackage{algorithm}
\usepackage{algpseudocode}
\newcommand*{\Positives}{\setN}

\newcommand{\R}{\mathcal{R}}
\newcommand{\CFP}{\chi^2_{CF}}
\newcommand{\CFPt}{\chi^t_{CF}}
\newcommand{\CF}{\chi_{CF}}
\DeclareMathOperator{\Del}{Del}
\newcommand*{\iinterval}[2]{\llbracket #1 \ldots #2 \rrbracket }
\DeclareMathOperator{\twr}{twr}

\newcommand{\E}{\mathcal{E}}
\newcommand{\D}{\mathcal{D}}
\newcommand\famB{\mathcal B}
\newcommand\famC{\mathcal C}
\newcommand\setN{\mathbb N}
\newcommand\setR{\mathbb R}

\newcommand{\chium}{\chi_{\text{UM}}}
\newcommand{\chitum}{\chi_{t\text{-UM}}}

\EventEditors{John Q. Open and Joan R. Access}
\EventNoEds{2}
\EventLongTitle{42nd Conference on Very Important Topics (CVIT 2016)}
\EventShortTitle{CVIT 2016}
\EventAcronym{CVIT}
\EventYear{2016}
\EventDate{December 24--27, 2016}
\EventLocation{Little Whinging, United Kingdom}
\EventLogo{}
\SeriesVolume{42}
\ArticleNo{23}

\begin{document}

\maketitle

\begin{abstract}
We introduce and study conflict-free colourings of $t$-subsets in hypergraphs. In such colourings, one assigns colours to all subsets of vertices of cardinality $t$ such that in any hyperedge of cardinality at least $t$ there is a uniquely coloured $t$-subset. 

The case $t=1$, i.e.,\@ vertex conflict-free colouring, is a well-studied notion that has been originated in the context of frequency allocation to antennas in cellular networks. It has caught the attention of researchers both from the combinatorial and algorithmic points of view. A special focus was given to hypergraphs arising in geometry.

Already the case  $t=2$ (i.e., \@ colouring pairs) seems to present a new challenge. Many of the tools used for conflict-free colouring of geometric hypergraphs rely on hereditary properties of the underlying hypergraphs. When dealing with subsets of vertices, the properties do not pass to subfamilies of subsets. Therefore, we develop new tools, which might be of independent interest.

 (i) For any fixed $t$, we show that the $\binom n t$ $t$-subsets in any set $P$ of $n$ points in the plane can be coloured with $O(t^2 \log^2 n)$ colours so that any axis-parallel rectangle that contains at least $t$ points of $P$ also contains a uniquely coloured $t$-subset.
 
 (ii) For a wide class of ``well behaved" geometrically defined hypergraphs, we provide near tight upper bounds on their $t$-subset conflict-free chromatic number.
 
  \noindent For $t=2$ we show that for each of those ``well -behaved" hypergraphs $H$, 
  the hypergraph $H'$ obtained by taking union of two hyperedges from $H$, admits a $2$-subset conflict-free colouring with roughly the same number of colours as $H$. For example, we show that the $\binom n 2$ pairs of points in any set $P$ of $n$ points in the plane  can be coloured with $O(\log n)$ colours such that for any two discs $d_1,d_2$ in the plane with $|(d_1\cup d_2)\cap P|\geq 2$ there is a uniquely (in $d_1 \cup d_2$) coloured pair.
 
 (iii) We also show that there is no general bound on the $t$-subset conflict-free chromatic number as a function of the standard conflict-free chromatic number already for $t=2$. 

\end{abstract}

\section{Introduction}
\label{sec:introduction}

A {\em hypergraph} $H$ is a pair $(V,\E)$ where  $\E\subseteq 2^V$. The elements of $V$ and $\E$ are referred to as {\em vertices} and {\em hyperedges} respectively. 

For $k\in \setN$, a {\em $k$-colouring} of the vertices of $H$ is a function $\varphi$ from $V$ to a set of cardinality $k$ e.g., the set $\{1,\ldots,k\}$.
The colouring is \emph{proper} if for every $h \in \E$ with $|h| \geq 2$ there exist $x,y \in h$ with $\varphi(x) \neq \varphi(y)$. This extends the definition of the classical proper colouring of a graph. 
Denote by $\chi(H)$ the least integer $k$ such that $H$ admits a proper $k$-colouring.

The following notion of CF-colouring is a further restriction of proper colouring:

\begin{definition}\label{def:cf-coloring}
A {\em conflict-free colouring} ({\em CF-colouring} for short) of $V$ is a colouring $\varphi$ such that for any nonempty $h \in \E$ at least one vertex  $x \in h$ is uniquely coloured, meaning $\varphi(y) \neq \varphi(x)$ for all $y \in h \setminus \{x\}$. Denote by $\CF(H)$ the least integer $k$ such that $H$ admits a CF-colouring with $k$ colours. Clearly $\CF(H)\geq \chi(H)$.
\end{definition}

This notion was introduced to model radio frequencies allocations to antennas while avoiding interference \cite{ELRS,SmPHD}. This spawned a new area of research in combinatorics and computational geometry, with dozens of followup papers and theses. For more on CF-colouring and its applications see the survey \cite{CF-survey} and the references therein. 

In this paper we study an extension of the notion of CF-colouring to subsets of vertices.

\begin{definition} Let $H=(V,\E)$ be a hypergraph and let $t\in \setN$.
A \emph{$t$-subset-CF-colouring} of $H$ with $k$ colours is a function $\varphi$ from $\binom{V}{t}$ to a $k$-element set such that for every hyperedge $h\in \E$ with $|h|>t$ there exists a $t$-subset $s \in \binom h t$ whose colour $\varphi (s)$ is distinct from all colours of other $t$-subsets in $\binom h t$. The \emph{$t$-subset-CF-chromatic number} $\chi^t_{CF}(H)$ is the least integer $k$ such that $H$ admits a $t$-subset-CF-colouring with $k$ colours.
\end{definition}

\subsection{Sparse hypergraphs}

We show that in hypergraphs $H$ with $n$ vertices exhibiting a certain kind of sparsity property, as is the case in many geometrically defined hypergraphs we have  $\CFPt(H) = O(\log n)$ where the constant of proportionality in the big-Oh notation depends on $t$ and the the sparsity parameter. To make this statement precise
we need the following definitions:

The {\em Delaunay graph} of a hypergraph $H=(V,\E)$ is the graph $\Del(H)=(V,\{h\in \E\colon |h|=2\})$. For a subset $V' \subset V$ the {\em induced sub-hypergraph} $H[V']$ is the hypergraph $(V', \{h \cap V': h \in \E\})$.

\begin{definition}
    A hypergraph $H=(V,\E)$ has the {\em Hereditary Linear Delaunay} (HLD) property with parameter $c \in \setR_{>0}$ if for every subset of vertices $V'\subset V$ the graph $\Del (H[V'])$ has at most $c|V'|$ edges.
\end{definition}

\begin{theorem}\label{thm:linear-delaunay}
Let $H=(V,\E)$ be a hypergraph. If $H$ has the HLD property with parameter $c\in \setR_{>0}$ then $\CF^t(H)=O(ct^2 \log |V|)$. Moreover, there are such hypergraphs for which $\CF^t(H)=\Omega(\log |V|)$.
\end{theorem}

In order to prove~\cref{thm:linear-delaunay} we first need to prove a common generalization of previous results on so-called {\em strong CF-colouring} \cite{ABG+05, hks09}. Then, we use it as an auxiliary colouring for subset-CF-colouring.

Next, we prove the following result on hypergraphs which are derived from ``well-behaved" hypergraphs by allowing union of hyperedges:
Given a hypergraph $H=(V,\E)$, define a new hypergraph  on the same vertex set by $H^{\cup} = (V, \{  e \cup f \colon e, f \in \E\})$.
\begin{theorem}
\label{thm:union}
	If $H=(V,\E)$ has the HLD property for some parameter $c\in\setR_{>0}$, then $\CFP(H^{\cup})=O(c\log |V|)$.
\end{theorem}

 Given two families of sets $\famB$ and $\famC$, the \emph{intersection hypergraph} of $\famB$ with respect to $\famC$ is the hypergraph $H(\famB,\famC)$ on the vertex-set $\famB$, in which any $c \in \famC$ defines a hyperedge $\{b \in \famB  \colon  b \cap c \neq \emptyset \}$.

Many geometrically defined hypergraphs (also called {\em range spaces})
admit the HLD property. This includes $H(P,\mathcal{D})$ where $P$ is a set of points and $\mathcal{D}$ is the set of all discs in $\setR^2$ (since $\Del(H(P,\mathcal D))$ is then the standard Delaunay graph, which is planar). This was also extended to $H(D,\mathcal{D})$ for a finite family of discs $D$ \cite{KellerSm20}. Another example is $H(P,\mathcal{H})$ where $P$ is a set of points and $\mathcal{H}$ is the set of all halfspaces in $\setR^3$ (by \cite[Lem.~3.1]{smoro}, combined with the fact that the union of two planar graphs on $n\geq 3$ vertices has at most $6n-12$ edges). 

Pseudo-disc families are another example:
\begin{definition}
	A finite set of simple closed Jordan regions in $\setR^2$ is a \emph{family of pseudo-discs} if the boundaries of any two regions in the family intersect at most twice.
\end{definition}

\begin{theorem}[\cite{Keszegh18}]\label{thm:Kesegh}
If $\famB$ and $\famC$ are families of pseudo-discs, then $\Del( H(\famB,\famC))$ is planar. 
\end{theorem}
In particular,  $H(\famB, \famC)$ has the HLD property with $c=3$.

For all these hypergraphs, denoting by $n$ the number of vertices, Theorem~\ref{thm:linear-delaunay} and Theorem~\ref{thm:union} give $\CFPt(H)=O(t^2 \log n)$ and $\CFP(H^{\cup})=O(\log n)$.

\subsection{Axis-parallel rectangles}
Consider  $H=H(P,\mathcal{R})$ where $P$ consists of $n$ points and $\mathcal{R}$ of axis-parallel rectangles, in $\setR^2$.

The problem of bounding $\CF(H)$ as a function of $n$ remains elusive after almost two decades of research; there is an exponential gap between the best known upper and lower bounds of $O(n^{0.368})$ and $\Omega(\frac{\log n}{\log^2 \log n})$ \cite{Chan12, ChenPST08}.

In strong contrast to this for $t\geq 2$ we prove the following result:

\begin{theorem}\label{thm:axis-parralel-rectangles}
Let $t\ge 2$ be an integer and let $P$ be a set of $n$ points and $\mathcal{R}$ a set of axis-parallel rectangles in the plane. Then $\chi^{t}_{CF}(H(P,\mathcal{R}))=O(t^2 \log^2 n)$.
\end{theorem}

Since such hypergraphs do not have linear Delaunay graphs, our proof of Theorem~\ref{thm:axis-parralel-rectangles} uses different approach and is more involved technically.

\subsection{General hypergraphs}
Contrary to the above results, in general there is no quantitative relation between $\CFPt$ and $\CF$.  For example, We show that there is no function $f$ such that $\CFPt(H) \leq f(\CF(H))$. 

\begin{theorem}
\label{thm:unbounded}
For any $t \geq 2$ there exists a sequence of hypergraphs $(H_i)_{i\in\setN}$ such that $\CF(H_i)=2$ and $\lim_{i\to + \infty}\CFPt(H_i) =  +\infty$.
\end{theorem}

Notice also that there is no function $f$ such that $\CF(H) \leq f(\CFPt(H))$ already for $t=2$. One can take for example, the complete graph $K_n$ and easily verify that $\CFPt(K_n)=1$ and $\CF(K_n)=n$ for any $n$. We also show that $\CF(H)$ can be much larger than $\CFPt(H)$, even in non-trivial geometrically-defined hypergraphs. 

\begin{theorem}
\label{thm:unbounded2}
For any $n \in \mathbb{N}$, there is a hypergraph $H=(V,\E)$ with $|V|=n$ such that $\CFP(H)=O(\log n)$ and $\CF(H)=n$. 
\end{theorem}

The phenomenon witnessed by Theorem~\ref{thm:unbounded2} can be realised by geometric intersection hypergraphs defined with respect to points and halfspaces in $\setR^4$ (\cref{subsec:interval_union}). 

\subsection{Organization of the paper}
In Section~\ref{sec:prelim} we present several additional definitions and facts that we use throughout the paper. In Section~\ref{sec:rectangles} we prove Theorem~\ref{thm:axis-parralel-rectangles} 

In Section~\ref{sec:sparse_del} we prove Theorem~\ref{thm:linear-delaunay}.
Finally, in Section~\ref{sec:negative-results} we prove Theorem~\ref{thm:unbounded} and Theorem~\ref{thm:unbounded2} and provide geometric realisations of the underlying hypergraphs. In that section we also prove Theorem~\ref{thm:union}.

\section{Preliminaries}
\label{sec:prelim}
Fix a hypergraph $H=(V,\E)$.
We need several more definitions of various hypergraph colurings.
We start with the definition of $t$-strong-CF-colouring extending that of CF-colouring, where we require at least $t$ uniquely coloured vertices in any hyperedge:
\begin{definition}[\cite{hks09}]
	Let $t \in \Positives$. A colouring of $H$ is a \emph{$t$-strong-CF-colouring} if each $h \in \E$ contains at least $\min \{|h|,t\}$ vertices whose colour is unique in $h$. Let $\chi_{t\text{-strong-CF}}(H)$ be the least number of colours required in a $t$-strong-CF-colouring of $H$.  
\end{definition}

Next, consider yet another extension of CF-colouring:
\begin{definition}[\cite{cheilarisUniqueMaximumConflictFreeColoring2013,CF-survey}]
	A \emph{unique-maximum colouring} of $H$ (UM-colouring for short) is a colouring of $V$ with an ordered set of colours (e.g.,\@ integers) in which the maximum colour in each hyperedge is also unique in the hyperedge. Let $\chium(H)$ be the least number of colours needed in a UM-colouring of $H$.
\end{definition}
 
Note that any UM-colouring of $H$ is also a CF-colouring.

Similarly to $t$-strong CF-colouring, we extend UM-colouring and define $t$-unique-maximum colourings:

\begin{definition}
	Let $t \in \Positives$. A colouring with an ordered set of colours is a \emph{$t$-UM-colouring} of $V$ if in any $h \in \E$, the $\min \{  |h|,t  \}$ largest colours are unique. Let $\chitum(H)$ be the least number of colours required in a $t$-UM-colouring of $H$. 
\end{definition}

The following definition is yet another extension of the standard notion of proper colouring of a hypergraph:

\begin{definition} [\cite{hks09}]
	Let $t\in \setN$. A colouring is $t$-\emph{colourful} if any hyperedge $h\in \E$ contains at least $\min \{|h|,t \}$ pairwise distinctly-coloured vertices. Let $\chi_{t\text{-colourful}}(H)$ be the least number of colours required in a $t$-colourful colouring. 
\end{definition}

Observe that for all $t \in \setN$ we have $\chitum \geq \chi_{t\text{-strong-CF}} \geq \chi_{t\text{-colourful}}$, since $t$-UM colourings are also $t$-strong-CF and $t$-strong-CF colourings are also $t$-colourful.

\subsection{Colouring meta-algorithm}

\begin{algorithm}[H]
	\caption{A general hypergraph colouring scheme}
	 \hspace*{\algorithmicindent} \textbf{Input:} 
	 $H=(V,\E)$, hypergraph colouring subroutine $\mathbf{aux}$\\ 
	 \hspace*{\algorithmicindent} \textbf{Output:} a colouring $\phi$ of $H$
	\label{alg:gen-CF-framework}
	\begin{algorithmic}[1] 
		\State $i\gets 1$ \Comment{$i$ denotes an unused colour}
		\While{$V\neq \emptyset$}
		\State $\phi \gets \mathbf{aux}(H[V])$ \Comment{auxiliary colouring}
		\State {$V' \gets$ largest colour class of $\phi$ (arbitrarily chosen in case of a tie)}
		\State colour all vertices of $V'$ with $i$
		\State $V\gets V\setminus V'$
		\State $i\gets i+1$
		\EndWhile
	\end{algorithmic}
\end{algorithm}

\paragraph*{Analysis of \cref{alg:gen-CF-framework}}
If the subroutine $\mathbf{aux}$ uses at most $f(n)$ colours on $n$-vertex hypergraphs, then the number of vertices remaining to be coloured after $i$ iterations is at most $u_i$, where $u_{i+1} = u_i \cdot (1 - 1 / f(u_i))$ and $u_0=n$ is the number of vertices of $H$. The number of iterations in \Cref{alg:gen-CF-framework}, which also equals the number of colours in its output, is at most $\min \{T\in \setN\colon u_T<1\}$.
For example, if $f(n)=O(1)$ then $T=O(\log u_0)$, if $f(n) = O(\log^{\alpha} n)$ for some positive $\alpha$ then $T = O(\log^{\alpha+1} u_0)$ and if $f(n)=O(n^\alpha)$ then $T = O(u_0^\alpha)$.

\Cref{alg:gen-CF-framework} was used in \cite{smoro} for finding a UM-colouring of a hypergraph, by using the subroutine \textbf{aux} to be a proper colouring. It was then proved that UM-colourings and proper colourings are strongly related as is summarised in the following theorem:

\begin{theorem}[\cite{smoro}]\label{thm:Proper-to-CF}
    For any finite hypergraph $H=(V,\E)$,
    \[\chium(H)
	= O\left( \max_{V'\subset V} \chi(H[V']) \cdot \log |V|\right).\]
\end{theorem}

In this paper we use the same meta-algorithm with other choices of auxiliary colourings that will be useful for our purposes.

\subsection{Linear Delaunay graphs}
We need one more technical lemma about hypergraphs with the HLD property. It was explicitly stated in \cite{AKP21}, based on earlier ideas from \cite[Lem.~2.6]{ADEP21} and \cite{BPR13}, and can be viewed as an abstract version of the Clarkson--Shor technique \cite{cs-arscg-89}:

\begin{lemma} (\cite[Lem.~22]{AKP21}) \label{lem:k_good_pairs}
	Let $H=(V,\E)$ be a hypergraph with the HLD property for some parameter $c\in \setR_{>0}$. Then, for any $k\in\setN$,
	\[ \left\lvert \left\{ \{v_1,v_2\} \in \binom{V}{2} \colon \exists h \in\E,\ |h| \leq k \land \{v_1,v_2\} \subset h \right\} \right\rvert \leq c |V|ek,
	\]
where $e$ is the base of the natural logarithm.
\end{lemma}

\section{Colouring subsets of points with respect to axis-parallel rectangles}\label{sec:rectangles}

In this section we prove \cref{thm:axis-parralel-rectangles}.

\begin{proof}[Proof of \cref{thm:axis-parralel-rectangles}]
Let $n$, $t$, $P$ and ${\cal R}$ be as in the statement of the theorem. We consider the intersection hypergraph $H=H(P,{\cal R})$. Note that since $P$ is finite, we can assume that the set ${\cal R}$ is finite without changing the set of hyperedges of $H$ and while maintaining the properties mentioned below. 

	We can assume, without loss of generality, that the points of $P$ lie on the $n\times n$ integer grid and furthermore that no two points in $P$ share an $x$- or $y$-coordinate. A small perturbation of $P$ and monotone transformations of the coordinates ensure both, without removing any hyperedge from $H(P,\R)$.
 
	Put $I= \{-\lceil \log n\rceil, \dots, \lceil \log n\rceil\}$. For each $i\in I$, let $A_i$ be the set of all axis-parallel rectangles with width-to-height ratio $2^i$. Each $A_i$ is a family of pseudo-discs\footnote{After an infinitesimal perturbation of each $r \in A_i$ which does not affect $r\cap P$.} (see \cite{ACPIN2013}), and the following properties hold:
	\begin{enumerate}
	    \item[(a)] For every $r \in \R$, there exist $i \in I$ and $r_1, r_2 \in A_i$ such that $r =r_1 \cup r_2$. \label{prop-cover}
	    Indeed, if $r$ has width $w$ and height $h\leq w$ then it can be covered by two rectangles of height $h$ and width $2^{\lfloor \log (w/h)\rfloor} h$ which are contained in $r$ (see \cref{fig:propA}). The case $w< h$ is symmetric.
    \item[(b)] \label{prop-access}For every $r\in A_i$ such that $|r\cap P|\ge t+1$, there exists $d\in A_i$ such that $d \subset r$ and $|d\cap P|=t+1$. Such a rectangle $d$ can be obtained from $r$ by a translation and scaling, while maintaining its width-to-height ratio.
    \end{enumerate}
    
\begin{figure}
\begin{center}
\includegraphics[width=0.4\textwidth]{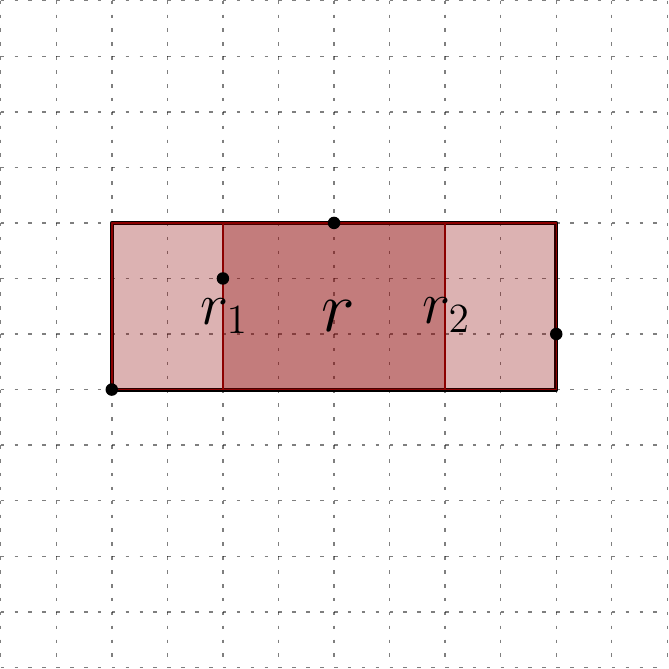}
\end{center}
\caption{A rectangle $r$ covered by $r_1$ and $r_2$. The rectangle $r$ is of height 3 and width 8, and each $r_i$ is of height 3 and width 6.}\label{fig:propA}
\end{figure}
	
	For $i\in I$, let 
	\[ E_i = \left\{ \{p,q\}\in \binom {d \cap P} 2 \colon d\in A_i, |d\cap P|\le t+1\right\}.\]
	\Cref{lem:k_good_pairs} together with \cref{thm:Kesegh} applied to $H(P, A_i)$ give $|E_i|\leq 3e \cdot (t+1) |P|$.
	Let $G:=(P,\cup_{i\in I} E_i)$. First, we prove by induction on $|P|=n$ that $\chi(G) \leq 80t\log n +1$. The graph $G$ has at most $3e\cdot (t+1)|P|\cdot |I| \leq 40tn\log n$ edges, hence $G$ has a vertex $p \in P$ of degree at most $80t \log n$. Let $P'= P \setminus \{p\}$, then by the induction hypothesis, the graph $G'$ defined like $G$ on $P'$, satisfies $\chi(G') \leq 80t \log n +1$. To argue by induction that $\chi(G) \leq 80 t \log n +1$ we need to verify that if $\{x,y\}\in E(G)$ and $x,y \neq p$, then $\{x,y\}\in E(G')$. Indeed, if $\{x,y\}\in E(G)$ then there is some $d\in A_i$ for some $i\in I$ such that $|d\cap P|\le t+1$ and $\{x,y\}\in \binom {d \cap P} 2$, then $|d\cap P'|\le |d\cap P|\le t+1$ and $\{x,y\}\in \binom {d \cap P'} 2$, hence $\{x,y\}\in E(G')$, as required.  
	Now, by colouring inductively $G'$ with at most $80 t \log n +1$ colours, and assigning $p$ a colour that differs from the colours of all its neighbours in $G$, we are done.

	The proper colourings of $G$ are, by Property ($b$), exactly the $(t+1)$-colourful colourings of $H(P,\cup_{i\in I} A_i)$. 
	
    Similarly we obtain a $(t+1)$-colourful colouring $\phi^{P'}$ in $H(P',\cup_{i\in I} A_i)$ for any $P'\subset P$ with $O(t \log |P'|)$ colours.
	
	Applying \cref{alg:gen-CF-framework} with $\phi^{P'}$ as the auxiliary colouring, we obtain a UM-colouring $c$ of $H(P,\cup A_i)$ with $O(t \log^2 n)$ colours. 
	
	\begin{claim}\label{claim:colourCount}
	Let $r\in \mathcal{R}$ and $P'$ be the set of points remaining after some iterations of \cref{alg:gen-CF-framework}. If $|r\cap P'| \leq t+2$, then each colour of $\phi^{P'}$ appears at most twice in $r \cap P'$, and if $|r \cap P'|= t+k$ (for $k \geq 3$), then each colour of $\phi^{P'}$ appears at most $k$ times in $r \cap P'$.
	\end{claim}
	
	\begin{proof}
	Let $i\in I$ and $r_1,r_2 \in A_i$ be such that $r=r_1\cup r_2$, as guaranteed by Property (a). Since the colouring $\phi^{P'}$ is $(t+1)$-colourful, each $r_i\cap P'$ ($i\in\{1,2\}$) receives at least $\min\{|r_i \cap P'|,t+1\}$ different colours. 

	If either of $r_1\cap P'$ or $r_2\cap P'$ is equal to $r\cap P'$ the conclusion follows trivially. Otherwise, $|r_1\cap P'|, |r_2\cap P'|\leq t+1$, so all colours appear at most once in $r_1$ and at most once in $r_2$. The second part is similar.

	\end{proof}
    
    \paragraph*{Deriving the final colouring for $\binom P t$}
    
    Recall the colouring $c$ of $P$ defined before \Cref{claim:colourCount}. We now define the colouring $\psi$ of $\binom P t$ as follows. For any $t$-subset $S \subset P$, if under $c$ some colour appears at least three times in $S$ then $\psi (S)=\bot$ (a dummy colour). Otherwise,
    denote by $r(S)$ the minimal axis-parallel rectangle such that $S\subseteq r(S)$. For a point $p\in P$ denote by $x(p)$ the $x$-coordinate of $p$.
    Also let $m = \min c(S)$, the least colour assigned to $S$ under $c$. 

    Define 
    \begin{equation} \label{eq:2t}
    \psi(S) = \left( \sum_{x \in S} c(x),Q(S)\right),
    \end{equation}
    where the value $Q(S)$ encodes the following information about $S$:\footnote{We write ``$m$ appears in $S$'', whenever there exists a point in $S$ which colour is $m$.}
    \begin{itemize}
        \item A. Whether $m$ appears exactly once, or exactly twice in $S$,
        \item B. Whether $m$ appears exactly once, exactly twice, or more than twice in $r(S)$,
        \item C. In the event where $m$ appears only once in $S$, its location in $r(S)$: on one of four open edges, at one of four corners, or in the interior of $r(S)$,
        \item D. In the event where $m$ appears only once in $S$ but exactly twice in $r(S)$, that is $p\in S$, $p'\in (r(S)\cap P)\setminus S$, and $c(p)=c(p')=m$, whether $x(p) < x(p')$.
    \end{itemize}
    While the precise encoding scheme is immaterial, $Q$ can be chosen to take a constant number of values. 
    Therefore, $\psi$ uses $O(t^2 \log^2 n)$ colours.

    \paragraph*{Correctness}
     Let us now check that $\psi$ is conflict-free, that is, that any rectangle $r\in {\cal R}$ such that $|r\cap P|\ge t+1$ contains a $t$-subset whose colour under $\psi$ is unique.
     
    Let $r\in \mathcal{R}$ and for every positive integer $\ell$ let $C_{\ell}$ be the set of points in $r\cap P$ receiving the $\ell$-th largest colour, in $r\cap P$, under $c$. Let $C_{\leq\ell}=\cup_{k\leq\ell} C_{k}$: the set of all points in $r \cap P$ which receive one of the $\ell$ largest colours, in $r\cap P$, under $c$. \Cref{claim:colourCount} can be restated as: if $|C_{\leq\ell}|\leq t + 2$ then $|C_1|, \dots, |C_\ell|\leq 2$, and if $|C_{\leq\ell}|= t + k$ for $k\geq 3$, then $|C_1|, \dots, |C_\ell|\leq k$.
        
    First, we consider the case where there is an index $\ell$ for which $|C_{\leq\ell}|=t$. In this case $C_{\leq\ell}$ is a $t$-subset with a sum of colours larger than the sum of any other $t$-subset of $r \cap P$ and hence uniquely coloured, by the first coordinate of \eqref{eq:2t}. 
    
    Next, we consider the complementary case where there is no index $\ell$ such that $|C_{\leq\ell}|=t$, then by \cref{claim:colourCount} there must be an index $\ell$ such that $|C_{\leq\ell}|=t+1$ and $|C_{\ell}|=2$. (Actually, \cref{claim:colourCount} implies that every $C_i$ for $1\leq i \leq \ell$ contains either 1 or 2 points.) Let $p\neq p'$ be the two points getting the smallest colours in $C_{\leq \ell}$ under $c$, namely, $m=c(p)=c(p')=\min\{c(C_{\leq \ell})\}$. We further divide cases based on how many points out of $\{p,p'\}$ appear on the boundary of $r(C_{\leq\ell})$. We argue that in all but one of the cases below the sets $S=C_{\leq\ell}\setminus \{p\}$ and $S'=C_{\leq\ell}\setminus \{p'\}$ have $Q(S)\neq Q(S')$. As those subsets have the maximal sum out of all $t$-subsets in $r\cap P$ (this sum is the first coordinate of \eqref{eq:2t}), they are uniquely coloured out of all the $t$-subsets in $r\cap P$, i.e. we get two uniquely coloured subsets instead of the required one. Note that $m=c(p)=c(p')$ appears only once in each of $S$ and $S'$. 
    \begin{itemize}
        \item 1. Neither $p$ nor $p'$ lie on the boundary of $r(C_{\leq\ell})$. Then $m$ appears twice in $r(S)=r(S')=r(C_{\leq\ell})$, so item (D) in the definition of $Q$ guarantees $Q(S)\neq Q(S')$.
        \item 2. Exactly one of $p$ and $p'$ lies on the boundary of $r(C_{\leq\ell})$, say $p$. Then $m$ appears twice in $r(S')$ but only once in $r(S)$, so $Q(S)\neq Q(S')$ (item (B)).
        \item 3. Both $p$ and $p'$ lie on the boundary of $r(C_{\leq\ell})$, on different edges as no two points have the same $x$- or $y$-coordinates. As long as $p$ and $p'$ do not appear on the same relative corner of $r(S')$ and $r(S)$, that is, not both $p$ and $p'$ appear at the top-right corner of $r(S')$ and $r(S)$, respectively, (and the same holds for the top-left, bottom-right and bottom-left corners), item (C) ensures $Q(S)\neq Q(S')$. Otherwise, 
        let $r=r(C_{\leq\ell})$ and consider the iteration of \cref{alg:gen-CF-framework} in which $r \cap P'=C_{\leq\ell}$. 
        In this iteration, $p$ and $p'$ were assigned the same colour under $c$, hence the same colour in the auxiliary colouring of the algorithm. 
        This means that in the covering $r=r_1\cup r_2$ obtained by property (a), without loss of generality, $p \in r_1 \setminus r_2$ and $p' \in r_2 \setminus r_1$. Since $p $ and $p'$ are in the same relative corner of $r(S')$ and $r(S)$, we have  $S'= C_{\leq\ell} \cap r_1$ or $S= C_{\leq\ell} \cap r_2$. Hence all points in $C_{\leq\ell} \setminus \{ p,p' \}$ get distinct colours under $c$.
        In this case, the set $S''$ which is $C_{\leq\ell}$ minus its point coloured with the smallest colour larger than $c(p)=c(p')$ is uniquely coloured under $\psi$. It has the maximal sum out of the $t$-subsets with the minimum appearing exactly twice (item (A)). 
    \end{itemize}
        See \cref{fig:rect} for an illustration of the different cases where $t=5$ and $|C_{\le \ell}|=6$. In each of the cases, the drawn rectangle is $r(C_{\le \ell})$. Note that the drawn rectangles can contain more points of $(P\cap r(C_{\le \ell}))\setminus C_{\le \ell}$ which are not depicted in the figure. 
        In \cref{fig:case3a}, the points $p$ and $p'$ are not on the same relative corner of $r(S)$ and $r(S')$. In \cref{fig:case3b}, the points $p$ and $p'$ are on the same relative corner of $r(S)$ and $r(S')$. Let $r_1,r_2$ be the rectangles which cover $r$ by \cref{prop-cover}. The points $p$ and $p'$ do not belong to the same $r_i$. Therefore all the other four points in $C_{\le \ell}$ belong to the same rectangle, which is $r_2$, so those points get distinct colours under $c$. In this case the uniquely coloured five-tuple $S''$ consists of $p,p'$ and the three maximal colours among the remaining four. 
\begin{figure}
\begin{subfigure}{.5\textwidth}
  \centering
  \includegraphics[width=.75 \linewidth]{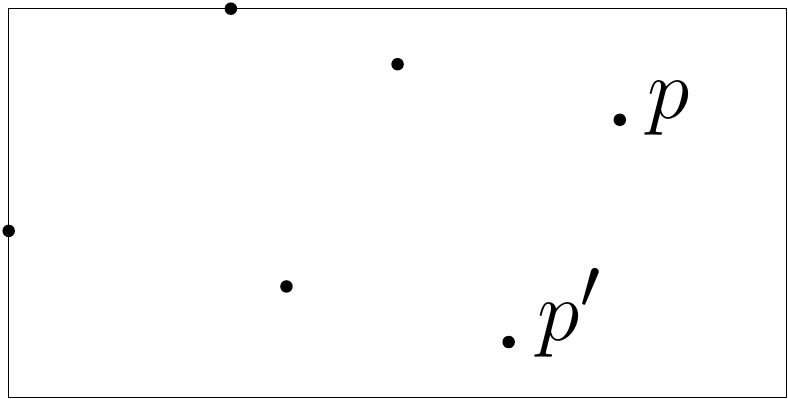}
  \caption{Case 1}
  \label{fig:case1}
\end{subfigure}%
\begin{subfigure}{.5\textwidth}
  \centering
  \includegraphics[width=.8\linewidth]{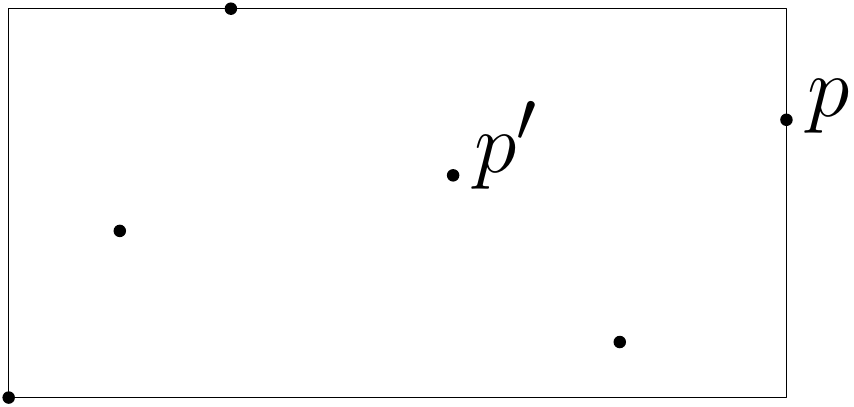}
  \caption{Case 2}
  \label{fig:case2}
\end{subfigure}
\newline
\begin{subfigure}{.5\textwidth}
  \centering
  \includegraphics[width=.8\linewidth]{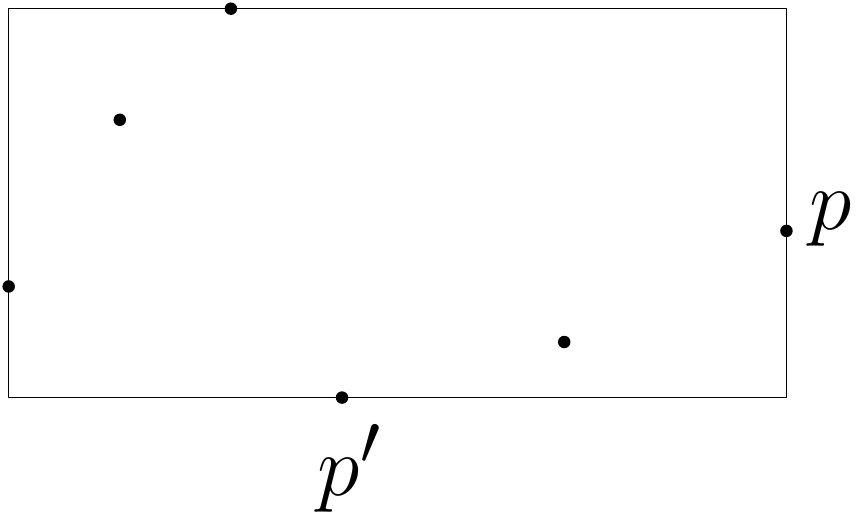}
  \caption{First possibility of case 3}
  \label{fig:case3a}
\end{subfigure}%
\begin{subfigure}{.5\textwidth}
  \centering
  \includegraphics[width=.8\linewidth]{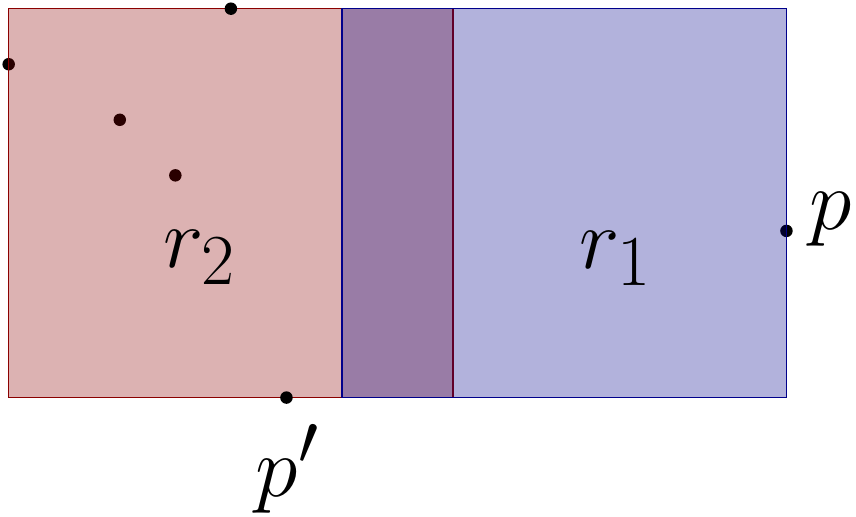}
  \caption{Second possibility of case 3}
  \label{fig:case3b}
\end{subfigure}
\caption{The different cases in the proof of \cref{thm:axis-parralel-rectangles}. In the first possibility of Case 3, $p$ and $p'$ are not on the same relative corner of $r(S)$ and $r(S')$. In the second possibility of Case 3, both $p$ and $p'$ are on the same relative corner (bottom right) of $r(S)$ and $r(S')$. Here in the covering $r=r_1 \cup r_2$, the points $p$ and $p'$ do not belong to the same $r_i$. Therefore, all other 4 points belong to the same $r_i$ (which is $r_2$ in the figure), hence they get distinct colours under $c$. In this case, the uniquely-coloured 5-tuple, $S''$, consists of $p,p'$ and 3 out of the 4 other points, removing the fourth point whose colour under $c$ is the minimal among the four.}
\label{fig:rect}
\end{figure}
\end{proof}

\section{Hypergraphs with linear Delaunay graphs}
\label{sec:sparse_del}
\subsection{General combinatorial considerations}

Horev et al. \cite{hks09} showed that using a $(t+1)$-colourful colouring subroutine in \cref{alg:gen-CF-framework} yields a valid $t$-strong-CF-colouring. However, a more careful analysis of this algorithm reveals that the colouring is, in fact, a $t$-UM-colouring. Hence, we have:
\begin{claim}\label{cl:colourful2um}
Let $t,k\in \setN$ and let $H=(V,\E)$ with $|V|=n$ where for any $V' \subset V$, $\chi_{(t+1)\text{-colourful}}(H[V']) \leq k$. Then $\chitum(H)= O( k \log n)$.
\end{claim}

Let $H=(V,\E)$ be a hypergraph where $V=\{v_1,\ldots,v_n\}$. Given a $t$-strong-CF-colouring $c$ of $V$ with $x$ colours define a colouring of $\binom V t$ by assigning every $t$-subset $S=\{v_{i_1},\ldots,v_{i_t}\}$ ($i_1 < i_2 \cdots < i_t$) of $V$ a colour $c(S)$ which is the sequence $(c(v_{i_1}),\ldots,c(v_{i_t}))$. It is easy to check that the resulting colouring is a $t$-subset-CF-colouring and uses at most $k^t$ colours. Hence, we have the following observation:

\begin{observation} \label{cl:t-strong}
	Let $H$ be a hypergraph with $\chi_{t\text{-strong-CF}}(H)=k$. Then $\CF^t(H)=O(k^t)$.
\end{observation} 

However, if the colouring $c$ of $V$ has the stronger property of being $t$-UM, the corresponding $t$-subsets CF-chromatic number is significantly smaller.

\begin{lemma}\label{lem:t-UM}
    	Let $H=(V,\E)$ be a hypergraph with $\chitum(H)=k$. Then $\CF^t(H)=O(tk)$. 
\end{lemma} 

\begin{proof}
Let $\varphi \colon V \rightarrow \{1,\ldots,k\}$ be a $t$-UM-colouring of $H$.
Define $\phi \colon \binom{V}{t} \rightarrow \{1,\ldots,tk\}$ by $\phi (S) = \Sigma_{v\in S} \varphi(v)$. Clearly, the maximum colour given under $\phi$ in each hyperedge of size at least $t$ is unique.
\end{proof}

All results on $t$-strong CF-colouring of which we are aware rely on \cref{alg:gen-CF-framework} (via $(t+1)$-colourful auxiliary colourings), and therefore the resulting colourings are in fact $t$-UM. However, this $t$-UM property has not been used explicitly before, neither in the context of $t$-strong-CF-colouring, nor in the more general context of CF-colouring. In the following subsection we exploit this property, applying \cref{lem:t-UM}.

\subsection{Geometric applications}

\begin{corollary}\label{cor:pdUM}
Let $\D$ be a collection of $n$ pseudo-discs and let $P$ be a set of points in the plane. Then $\CF^t(H({\cal D},P))=O(t^2 \log n)$.  
\end{corollary}

\begin{proof}
It is known that $H({\cal D},P)$ admits a $t$-UM-colouring with $O(t \log n)$ colours\footnote{The more general result in \cite{hks09} applies to families with linear union complexity. It is well known that this includes pseudo-disc families.} \cite{hks09}. By \cref{lem:t-UM}, the bound follows. 
\end{proof}

Similarly, since the intersection hypergraph of a family of $n$ axis-parallel rectangles with respect to points admits a $t$-UM-colouring with $O(t \log^2  n)$ colours \cite{hks09}, we have:

\begin{corollary}\label{cor:recCF}
Let $\R$ be a collection of $n$ axis-parallel rectangles and let $P$ be a set of points in the plane. Then $\CF^t(H({\cal R},P))=O(t^2 \log^2 n)$.
\end{corollary} 

\subsection{Colouring \texorpdfstring{$t$}{t}-subsets in hypergraphs with linear Delaunay graph}

In this section we prove \cref{thm:linear-delaunay}. Note that \cref{cor:pdUM} above is an immediate consequence of \cref{thm:linear-delaunay} since, by \cref{thm:Kesegh}, the corresponding hypergraphs have a HLD property, while \cref{cor:recCF} does not follow from \cref{thm:linear-delaunay}. In fact, for the corresponding hypergraphs their Delaunay graphs may have a quadratic number of edges. An easy example follows from drawing $n/2$ horizontal line segments and $n/2$ vertical line segments so that their Delaunay graph is the complete bipartite graph $K_{\frac{n}{2},\frac{n}{2}}$.

The proof of \cref{thm:linear-delaunay} uses \cref{lem:t-UM}. Therefore, we start by showing that for a wide class of hypergraphs $H$ there is a $t$-UM-colouring with a `small' number of colours:

\begin{theorem}\label{thm:Delaunay}
	Let $H=(V,\E)$ be a hypergraph with the HLD property for some $c\in\setR_{>0}$. Then $\chitum(H)=O(c t \log |V|)$. 
\end{theorem}

\cref{thm:Delaunay} generalises previous results for the intersection hypergraph of pseudo-discs and points \cite{hks09}, and for the intersection hypergraph of points and discs\footnote{The proof in \cite{ABG+05} relies on disc-specific arguments that fail for pseudo-discs.} \cite{ABG+05}.

\begin{proof}[Proof of \cref{thm:Delaunay}]
	By \cref{cl:colourful2um}, it suffices to prove that any induced sub-hypergraph $H[V']$ on some $V' \subset V$ admits a $(t+1)$-colourful colouring with at most $2cet+1$ colours. Such a colouring can be obtained
	by induction on $|V|$. Assume we proved the assertion for hypergraphs with $|V|-1$ vertices. 
	Define a graph $G_t$ in which we connect two vertices if they participate together in some hyperedge of size at most $t$ of $H$. 
	
	By \cref{lem:k_good_pairs}, the number of edges in $G_t$ is at most $c|V|et$, hence the average degree of $G_t$ is at most $2cet$. So there must be a vertex $v \in V$ whose degree in $G_t$ is at most $2cet$. This means that $v$ belongs to at most $cet$ hyperedges each of which has cardinality at most $t$. By the induction hypothesis, $H[V \setminus \{v\}]$ admits a $(t+1)$-colourful colouring with at most $2cet+1$ colours. (Note that $H[V \setminus \{v\}]$ satisfies the conditions of the theorem.) Assign to $v$ a colour that differs from the colours of all its neighbours in $G_t$. To see that the resulting colouring is a valid $(t+1)$-colourful colouring consider a hyperedge $h$. We can assume that $v \in h$ for otherwise $h$ satisfies the required conditions by the induction hypothesis.
	Also if $|h| > t$ we are done by the induction hypothesis as $h\setminus\{v\}$ already contains the required number of colours. Otherwise if $|h| \leq t$ then by the induction hypothesis $h\setminus\{v\}$ consists of pairwise distinct coloured vertices. Since $v$ is a neighbour in $G_t$ of any vertex in $h\setminus\{v\}$ it is assigned a distinct colour. This completes the induction step. 
\end{proof}

We are ready to prove \cref{thm:linear-delaunay}.

\begin{proof}[Proof of \cref{thm:linear-delaunay}]
By \cref{thm:Delaunay}, $H$ admits a $t$-UM-colouring with $O(c t \log n)$ colours. In view of \cref{lem:t-UM}, it follows that $\CF^t(H)=O(c t^2 \log n)$.

As for the tightness of this bound, consider for each $n\in\Positives$ the hypergraph $H_n=H(P,{\cal I})$ where $P=[n]$ is the set $\{1,\ldots,n\}$ and ${\cal I}$ is the family of all intervals on the real line. 
Clearly, $H_n$ has the HLD property (with $c=1$).
Let $\varphi$ be a $t$-subset-CF colouring of $H_{2n+1}$. Since there is an interval $I\in {\cal I}$ such that $I\cap P=P$, $I$ contains a uniquely-coloured $t$-subset $S$. Hence there exists $I'\in {\cal I}$ such that $I'$ does not contain $S$ and $|I'|\ge n$. Let $P'=I'\cap P$. The restriction of $\varphi$ to $H_{2n+1}[P']$ is still a $t$-CF-colouring but does not use the colour $\varphi(S)$. Hence the recurrence relation $\chi^t_{CF}(H_{2n+1}) \geq 1 + \chi^t_{CF} (H_n)$ is satisfied and therefore $\chi^t_{CF}(H_n)= \Omega (\log n)$.
\end{proof}

\section{Comparing \texorpdfstring{$t$}{t}-subset-CF-colouring and CF-colouring}
\label{sec:negative-results}

\subsection{\texorpdfstring{$\CFPt$}{\textchi\^{}t\_{CF}} is not bounded in terms of \texorpdfstring{$\CF$}{\textchi\texttwosuperior\_{CF}}}\label{subsec:comparing}

In this subsection we prove Theorem \ref{thm:unbounded} and show that there exist hypergraphs with bounded $\CF$ but arbitrarily large $\CFPt$. 

\begin{proof}[Proof of \cref{thm:unbounded}]

Consider $t,n\in \mathbb{N}$ such that $n\ge t \ge 2$ and let $H_n=(V,\E)$ with $V=\{1,2,\ldots,n\}$ and $\E=\{\{1\}\cup s \mid s\in {\binom{\{2,3,\ldots,n\}}{t+1}}\}$.

It is easy to see that  $\CF(H_n)=2$, with two colour classes $\{1\}$ and $\{2,3,\ldots,n\}$.

On the other hand, for any $k \in \setN$, there exists $n\in \mathbb{N}$ such that $\CFPt(H_n)>k$. Let
\[
n = \twr_{t-1}(3(b-t+1) k \log k )+1,
\]
where $b=\twr_t(3k\log k)$, and the {\em tower function} is defined by $\twr_1(m)=m$ and $\twr_t(m)=2^{\twr_{t-1}(m)}$.

Indeed, by the multicolour hypergraph Ramsey theorem~\cite{ER52}, for any $2 \leq r<s$ and $N \geq   \twr_r ( 3(s-r) k \log k ) $, in any $k$-colouring of all the $r$-subsets of an $N$-element set there exists a monochromatic $s$-subset, namely an $s$-subset all of whose $r$-element subsets have the same colour.

Assume to the contrary that $\varphi$ is a $t$-subset-CF colouring of $V$ with $k$ colours.
Define a colouring $\psi$ of $\binom{     \{  2,3,\ldots,n \}       }{t-1}$
by $\psi(\{i_1,i_2,\ldots,i_{t-1}\})      =   \varphi(\{1,i_1,i_2,\ldots,i_{t-1}\}) $.

Since $n-1 = \twr_{t-1}(3(b-t+1) k \log k)$, there exists some $b$-subset $B$ of $\{ 2,  \ldots ,n  \}$, all of whose $(t-1)$-subsets are assigned the same colour by $\psi$. In other words, all $t$-subsets of type $\{\{1,i_1,i_2,\ldots,i_{t-1}\}\colon i_1,i_2,\ldots,i_{t-1} \in B\}$ are assigned the same colour in $\varphi$.

Since $|B|=b=\twr_t(3k\log k)$, using the same result as above, we can deduce that there is a subset $B' \subset B$, with $|B'|=t+1$, all of which $t$-subsets receive the same colour under $\varphi$. Hence, $e=B' \cup \{1\}$ is a hyperedge of $H$ with no uniquely coloured $t$-subset under $\varphi$. Indeed, all $t$-subsets of $e$ containing 1 get the same colour, and all $t$-subsets that are contained in $B'$ get the same (possibly different) colour, a contradiction.
\end{proof}

\subsection{\texorpdfstring{$\CFP$}{\textchi\texttwosuperior\_{CF}} can be much smaller than \texorpdfstring{$\CF$}{\textchi\_{CF}}}

\subsubsection{Union of hyperedges}
\label{subsec:union}
Here  we prove \Cref{thm:union}.

\begin{proof}[Proof of Theorem~\ref{thm:union}]
By \cref{thm:Delaunay}, $H$ admits a 2-UM-colouring $\psi$ with $O(c\log |V|)$ colours. Define the following pairs-colouring:

\begin{align*}
    \phi \colon \{x,y\} \in \binom{V}{2} &\mapsto \left(\psi(x)+\psi(y), \begin{cases}0 \text{ if $\psi(x)=\psi(y)$,} \\ 1 \text{ otherwise} \end{cases}\right).
\end{align*} 

Consider some hyperedge $h = h_1 \cup h_2$ of $H^{\cup}$ where $ h_1,h_2 \in \E$ with $|h| \geq 3$.
Let $v_1,v_2,v_3$ be three vertices in $h$ with the top three maximal values of $\psi(h)$ so $\psi(v_1)\geq \psi(v_2)\geq \psi(v_3)$. Then (even without considering the second coordinate of $\phi$), the only case in which $\{v_1,v_2\}$ is not the uniquely-coloured pair is when $\psi(v_1)> \psi(v_2)= \psi(v_3)$. In this case, $\psi(v_2)= \psi(v_3)$ appears exactly once in each of $h_1,h_2$ (in one as the maximal colour, in the other as the second-maximal colour), and $\{v_2,v_3\}$ is  the uniquely-coloured pair -- here we use the second coordinate of $\phi$. 
\end{proof}

\subsubsection{Points on a line with respect to interval unions}\label{subsec:interval_union}
In some geometric hypergraphs, the constants of \cref{thm:union} can be improved:

\begin{theorem}\label{thm:union_intervals}
Let $H=H(P,\mathcal I)$ with $V$ a set of $n$ points on a line $\ell$ and $\mathcal I$ the set of intervals of $\ell$. Then $\CFP(H^{\cup}) \leq 4\log n$.
\end{theorem}

\begin{remark}
The hypergraph $H^{\cup}$ in \cref{thm:union_intervals} has a geometric interpretation as (isomorphic to) a hypergraph whose vertices are points on the moment curve in $\setR^4$, and whose hyperedges are induced by intersections with half-spaces. Indeed, a hyperplane in $\setR^4$ intersects the moment curve at most 4 times, and for any choice of up to 4 points on the moment curve in $\setR^4$ there exists a hyperplane that intersects the moment curve exactly at these points.
\end{remark}

\begin{proof}
Without loss of generality, the vertex set is $V=\{1,2,\ldots,n\}$ and $n=2^s-1$ for some integer $s=\lceil \log n\rceil$. Build a vertex colouring $\psi\colon V \to \{1,\dots,s\}$ as follows. The midpoint $(n+1)/2=2^{s-1}$ receives $\psi(2^{s-1})=s$. Each of the two halves (from $1$ to $2^{s-1}-1$ and from $2^{s-1}+1$ to $2^s -1$) has $2^{s-1} -1$ elements, colour them recursively with the colours from $1$ to $s-1$. It is easy to see that $\psi$ is a UM-colouring of $H$.

As for the pair colouring, define 
\begin{align*}
    \phi &\colon \{i,i+1\} \mapsto \left( \max \{\psi(i), \psi(i+1)\}, \begin{cases}0 \text{ if $\psi(i)<\psi(i+1)$,} \\ 1 \text{ otherwise} \end{cases}\right),
    \intertext{and for non-adjacent vertices (that is, $|i-j|>1$):}
    \phi &\colon \{i,j\} \mapsto (\psi(i)+ \psi(j), 0).
\end{align*} 
To check that $\phi$ is a pairs-CF colouring of $H^{\cup}$, we have to consider two types of hyperedges: those consisting of a single integer interval (of the form $\iinterval i j$ with $i<j$) and those consisting of two disjoint intervals (of the form $\iinterval i j\cup \iinterval k l$  where $i \leq j <k \leq l$). In the first case, $\psi$ has a unique maximum $m$ over $\iinterval i j$ and hence at least one of the pair-colours $(m,0)$ and $(m,1)$ is present exactly once in $\phi$. In the second case, $\psi$ attains respective maxima $m$ and $m'$ on each of the intervals. If $m\neq m'$ then at least one of the colours $(\max\{m,m'\},0)$ and $(\max\{m,m'\},1)$ appears exactly once, whereas if $m=m'$ then the colour $(2m,0)$ is present exactly once in the hyperedge.
\end{proof}

In \cref{subsec:comparing} we described a hypergraph whose $t$-subset CF-chromatic number is much larger than its CF-chromatic number. The hypergraph $H^{\cup}$ in \cref{thm:union_intervals} demonstrates the opposite situation: $\CFP(H^{\cup})=O(\log n)$, while
clearly $\CF(H^{\cup})=n$ and hence proves Theorem \ref{thm:unbounded2}.

However, this is not very informative since the underlying lower bound
is due to the hyperedges of size $2$, which we disregard for the pairs-CF colourings. It is therefore natural to consider the CF-chromatic number of $H^{\cup}$ restricted to its hyperedges of size at least $3$.

\cref{thm:LBunion} below shows that even if we consider only hyperedges of size at least 3 the same phenomenon holds -- the pairs-CF chromatic number of $H^{\cup}$ is still at most logarithmic in the CF-chromatic number of $H^{\cup}$.

\begin{theorem}\label{thm:LBunion}
Let $H^{\cup}$ be as in \cref{thm:union_intervals}. The sub-hypergraph $H' \subset H^{\cup}$ containing all hyperedges of size at least $3$ has
\[\CF(H') \geq \sqrt{n-1}.\]
\end{theorem}

\begin{proof}
Consider a CF-colouring of $H'$ with $\CF(H')$ colours. For every ordered pair of colours $(a,b)$ there is at most one pair of consecutive vertices that are coloured with $a$ and $b$ in this order, for otherwise there would be a hyperedge of size $4$ where both colour $a$ and $b$ appear twice.
Since there are $n-1$ consecutive pairs of vertices and $\CF(H')^2$ ordered pairs of colours, we have $\CF(H')^2 \geq n-1$ and the inequality follows.
\end{proof}

\bibliography{references}

\end{document}